\documentclass[12pt]{amsart}
\usepackage{amssymb,amsmath,amsthm,comment}
\usepackage{pdfpages,graphicx} 
\usepackage[section]{placeins} 
\usepackage{wrapfig}
\usepackage{float}
\usepackage{lineno}
\usepackage{youngtab}

\usepackage{setspace}
\newcommand{\shrinkmargins}[1]{
  \addtolength{\textheight}{#1\topmargin}
  \addtolength{\textheight}{#1\topmargin}
  \addtolength{\textwidth}{#1\oddsidemargin}
  \addtolength{\textwidth}{#1\evensidemargin}
  \addtolength{\topmargin}{-#1\topmargin}
  \addtolength{\oddsidemargin}{-#1\oddsidemargin}
  \addtolength{\evensidemargin}{-#1\evensidemargin}
  }
\shrinkmargins{0.5}

\newtheorem{theorem}{Theorem}
\newtheorem*{lemma}{Lemma}
\newtheorem*{conj}{Conjecture}
\newtheorem*{corollary}{Corollary}
\newtheorem*{theorem*}{Theorem}
{Claim}

\theoremstyle{definition}

\theoremstyle{remark}

\newtheorem*{remark}{Remark}

\numberwithin{theorem}{section} \numberwithin{equation}{section}

\usepackage{caption}
\usepackage{multirow}

\def\func#1{\mathop{\rm #1}}%
\def\pmatrix#1{\left(
#1
\right)}

\begin{document}
\title{Inequalities for $k$-regular partitions}
\author{Bernhard Heim}
\address{Department of Mathematics and Computer Science\\Division of Mathematics\\University of Cologne\\ Weyertal 86--90 \\ 50931 Cologne \\Germany}
\address{Lehrstuhl A f\"{u}r Mathematik, RWTH Aachen University, 52056 Aachen, Germany}
\email{bheim@uni-koeln.de}
\author{Markus Neuhauser}
\address{Kutaisi International University, 5/7, Youth Avenue,  Kutaisi, 4600 Georgia}
\address{Lehrstuhl A f\"{u}r Mathematik, RWTH Aachen University, 52056 Aachen, Germany}
\email{markus.neuhauser@kiu.edu.ge}
\begin{abstract}
We build upon
the work by
Bessenrodt and Ono, as well as
Beckwith and Bessenrodt concerning
the combined additive and multiplicative
behavior of
the $k$-regular partition functions $p_k(n)$. 
Our focus is on addressing
the solutions of the
Bessenrodt--Ono inequality 
\begin{equation*}
p_k(a) \, p_k(b) > p_k(a+b).
\end{equation*}
We determine the
sets $E_k$ and $F_k$ consisting
of all pairs $(a,b)$, 
where we have equality or the opposite inequality.
Bessenrodt and Ono previously
determined
the exception sets $E_{\infty}$ and $F_{\infty}$ for the partition function $p(n)$. 
We prove by induction that $E_k=E_{\infty}$ and $F_k=F_{\infty}$ if and only if $k \geq 10$. Beckwith and Bessenrodt
used analytic methods to
consider
$2 \leq k \leq 6$,
while Alanazi, Gagola,
and Munagi
studied the case $k=2$
using combinatorial methods.
Finally, we present
a precise and comprehensive conjecture on the log-concavity of the $k$-regular partition function
extending previous speculations
by Craig and Pun. The case $k=2$ was recently proven by Dong and Ji.
\end{abstract}

\maketitle
\section{Introduction}
In this paper we investigate
partition functions \cite{An98, On03} and
provide a complete solution for the Bessenrodt--Ono inequality \cite{BO16} for $k$-regular partition functions $p_k(n)$ \cite{Ha71}. We extend the
previous results for $2 \leq k \leq 6$
by Beckwith and Bessenrodt \cite{BB16} to $k \geq 2$,
proving
that for $k \geq 10$, the solutions
coincide with the solutions 
of the partition function $p(n)$. 

Let $\{\alpha(n)\}_n$ denote a sequence of positive integers. The sequence
satisfies the Bessenrodt--Ono inequality for a pair $(a,b)$ of positive integers if
\begin{equation}\label{BO}
\Delta(a,b):= \alpha(a) \, \alpha(b) - \alpha(a+b)
>0.
\end{equation}
For partitions $\alpha(0)=\alpha(1)=1$ and due to symmetry,
we only consider pairs with $1 < a \leq b$. 
Let $E$ denote the set of pairs
where $\Delta(a,b)=0$ and $F$ the set, where
$\Delta(a,b)<0$. For $k$-regular partition functions,
we use the symbols
$\Delta_k(a,b)$, $E_k$ and $F_k$, where $k \in \{2,3,4, \ldots \} \cup \{\infty\}$.

Let us recall the definition of the $k$-regular partition function $p_k(n)$.
A partition of a positive integer $n$ is a finite nonincreasing sequence of positive integers
that sum
to $n$. These integers are called the parts of the partition. The partition function $p(n)$ is the number of partitions of $n$ (\cite{An98}, Section 1.1). Let $k\geq 2$. The $k$-regular partition function $p_k(n)$ is the number of partitions of $n$ in which none of the parts
is a multiple
of $k$. Note, that $p_k(n)$ is also equal to the number of partitions
with each part
appearing at most $k-1$ times. Therefore, the number
of partitions of $n$ with odd parts is equal to the number of partitions with different parts. Based on the results of this paper and
the fact that $p_k(n) = p(n)$ for $k>n$,
we define
$p_{\infty}(n):=p(n)$.

In 2016,
Bessenrodt and Ono \cite{BO16} discovered the following
property of the partition function $p_{\infty}(n)$.
\begin{theorem}[Bessenrodt, Ono 2016]
Let $1 < a \leq b$. Then the Bessenrodt--Ono inequality
(\ref{BO}) is satisfied for the partition function $p_{\infty}(n)$,
except for the pairs 
\begin{eqnarray*}
E_{\infty} &=& \left\{ \left( 2,6\right) , \left( 2,7\right) , \left( 3,4\right) \right\} ,\\
F_{\infty} &=& \{(2,2), (2,3), (2,4),  (2,5), (3,3),(3,5)\}.
\end{eqnarray*}
\end{theorem}

Bessenrodt and Ono provided
an analytic proof (\cite{BO16}, Section 2) utilizing 
deep results by
Lehmer \cite{Le39} and Rademacher \cite{Ra37} for
the partition function,
built
on the circle method by Hardy and Ramanujan.
A second combinatorial proof was found by Alanazi, Gagola,
and Munagi \cite{AGM17}.
Their sophisticated proof involved
an injective map and an induction.
A third algebraic proof was provided in \cite{HN22}
using induction, and was
also generalized to plane partitions \cite{HNT23}.

Beckwith and Bessenrodt \cite{BB16} demonstrated that
$k$-regular partition functions
for $2 \leq k \leq 6$ also
fulfill the Bessenrodt--Ono inequality
(see Section \ref{regular}
for their result).

Moreover, many more sequences turned out to satisfy (\ref{BO}).
Chern, Fu, and Tang \cite{CFT18} generalized and proved the Bessenrodt--Ono
inequality (BO)
for $k$-colored partitions.
Hou and Jagadeesan \cite{HJ18}
extended on the numbers of partitions with ranks in a given residue class modulo $3$ and Males \cite{Ma20} did so
for general modulus. A
modified BO was presented
in \cite{HN19}.
Dawsey and Masri \cite{DM19} proved BO for the Andrews {\it spt}-function, and Gajdzica
did so for $A$-partitions.
In \cite{HNT20},
the polynomization of the $k$-colored partitions \cite{CFT18} was
initiated, leading to
a continous version of the Chern, Fu, Tang result \cite{CFT18} for
$k$-regular partitions.
The proof is utilizing an Euler-type recursion formula for the involved polynomials and their derivatives.
This polynomization approach and the underlying proof strategy
were also
successfully applied by Li \cite{Li23} for the overpartition function and
by Gajdzica et al.\ (\cite{GHN24}, Theorem 1.3)
for $A$-partition functions.

Recently,
several interesting and important papers
have explored the asymptotic
behavior leading to log-concavity properties and the 
Bessenrodt--Ono inequality,
as well as their relations. We refer to 
Bringmann, Franke,
and Heim (\cite{BFH24}, Theorem 1.3 and Theorem 1.5),
Benfield and Roy (\cite{BR24}, Theorem 1.1) and Gajdzica, Miska,
and Ulas (\cite{GMU24}, Theorem 2.1 and Theorem 3.1).

\section{Results on $k$-regular partitions}\label{regular}

Let $k \geq 2$. The generating function for the number of $k$-regular partitions is given by
\begin{equation*}
\sum_{n=0}^{\infty} p_k(n) \, q^n = \prod_{n=1}^{\infty} \frac{\left(1-q^{kn}\right)}{\left(1 -q^n\right)}.
\end{equation*}
Note that 
\begin{equation*}
\prod_{n=1}^{\infty} \frac{\left(1-q^{2n}\right)}{\left(1 -q^n\right)} =
\prod_{n=1}^{\infty} \left(1 -q^{2n-1}\right)^{-1},
\end{equation*}
which
implies that the number of $2$-regular partitions of $n$ is equal to the number
of partitions of $n$ with odd parts. We also refer to Glaisher's generalization for arbitrary $k$ (\cite{An98}, Corollary 1.3).

Let $g_k(n):= \sigma(n) - k \sigma(n/k)$, where $\sigma(n):= \sum_{d \mid
n} d$ and $\sigma(x)=0$ if $x \not\in \mathbb{N}$. Then
\begin{equation*}
\sum_{n=0}^{\infty} p_{k}\left( n\right) q^{n}= \exp \left( \sum_{n=1}^{\infty} g_k(n) \, \frac{q^n}{n}\right).
\end{equation*}
This follows from the identity $\prod_{n=1}^{\infty} \left(1 -q^n\right) = \exp \left( - \sum_{n=1}^{\infty} \sigma(n) \, q^n/n\right)$
and implies 
\begin{equation*}
n \, p_k(n) = \sum_{\ell =1}^n g_k(\ell) \, p_k(n - \ell),
\label{recurrence}
\end{equation*}
with initial value $p_k(0)=1$.
Beckwith and Bessenrodt proved,
in the same spirit as \cite{BO16},
that the $k$-regular partition function for $2 \leq k \leq 6$ satisfies (\ref{BO}).
They utilized a Rademacher-type formula for $p_k(n)$ given by Hagis \cite{Ha71} and determined for
each $k$ the exception sets $E_k$ and $F_k$.
We recall the following:
Let $a,b,k$ be positive integers with $1<a\leq b$ and let $k \geq 2$.
Then we denote by $E
_k$ the set of all pairs $(a,b)$ with equality in
(\ref{BO}) and $F
_k$ the set of all pairs $(a,b)$ with
\begin{equation*}
p_k(a) \, p_k(b) < p_k(a+b).
\end{equation*}

\begin{table}[H]
\begin{tabular}{|rp{66mm}|p{66mm}|}
\hline
$k$ & Elements of $E_k$
&  Elements of $F_k$
\\ \hline \hline
$2$&$\pmatrix{ 3,3}
$, $\pmatrix{ 3,5}
$, $\pmatrix{ 3,6}
$, $\pmatrix{ 3,7}
$, $\pmatrix{ 3,8}
$, $\pmatrix{ 4,15}
$, $\pmatrix{ 4,16}
$, $\pmatrix{ 4,17}
$, $\pmatrix{ 5,6}
$, $\pmatrix{ 5,7}
$, $\pmatrix{ 5,8}
$&$\left( 2,b\right) $, $b\geq 2$,
$\pmatrix{ 3,4}
$, $\pmatrix{ 4,4}
$, $\pmatrix{ 4,5}
$, $\pmatrix{ 4,6}
$, $\pmatrix{ 4,7}
$, $\pmatrix{ 4,8}
$, $\pmatrix{ 4,9}
$, $\pmatrix{ 4,10}
$, $\pmatrix{ 4,11}
$, $\pmatrix{ 4,12}
$, $\pmatrix{ 4,13}
$, $\pmatrix{ 4,14}
$, $\pmatrix{ 5,5}
$\\
$3$&$\pmatrix{ 2,2}
$, $\pmatrix{ 3,10}
$&$\pmatrix{ 2,3}
$, $\pmatrix{ 3,3}
$, $\pmatrix{ 3,4}
$, $\pmatrix{ 3,5}
$, $\pmatrix{ 3,6}
$, $\pmatrix{ 3,7}
$, $\pmatrix{ 3,8}
$, $\pmatrix{ 3,9}
$, $\pmatrix{ 3,11}
$, $\pmatrix{ 3,13}
$\\
$4$&$\pmatrix{ 2,2}
$, $\pmatrix{ 2,3}
$, $\pmatrix{ 2,5}
$, $\pmatrix{ 3,3}
$, $\pmatrix{ 3,4}
$, $\pmatrix{ 4,4}
$&$\pmatrix{ 2,4}
$\\
$5$&$\pmatrix{ 2,3}
$, $\pmatrix{ 2,4}
$&$\pmatrix{ 2,2}
$, $\pmatrix{ 2,5}
$, $\pmatrix{ 3,3}
$, $\pmatrix{ 3,5}
$\\
$6$&$\pmatrix{ 2,4}
$, $\pmatrix{ 2,5}
$, $\pmatrix{ 2,6}
$&$\pmatrix{ 2,2}
$, $\pmatrix{ 2,3}
$, $\pmatrix{ 3,3}
$\\
\hline
\end{tabular}
\caption{\label{ausnahmetabelle}Exceptions to the Bessenrodt--Ono inequality\\ for $2 \leq k \leq 6$}
\end{table}
\begin{theorem}[Beckwith, Bessenroth 2016]
Let $a,b,k$ be integers with $1<a \leq b$ and $2 \leq k \leq 6$.
Then the Bessenrodt--Ono inequality 
\begin{equation}
\label{pk}p_k(a) \, p_k(b) > p_k(a+b)
\end{equation}
is satisfied for all pairs,
except for the pairs presented
in Table \ref{ausnahmetabelle}
for $E
_k$ and $
F_k$.
Therefore, let $n_k$ and $m_k$ be given
by the following table:
\begin{equation*}
\begin{array}{|c|c|c|c|c|c|}
\hline
k   & 2 &  3 & 4 & 5& 6\\ \hline
n_k & 3 & 2 & 2 &2 &2 \\
m_k & 22 & 17 & 9 & 9 & 9 \\ \hline
\end{array}
\end{equation*}
Then for any $a,b \geq n_k$ and $a+b \geq m_k$, the Bessenrodt--Ono inequality
(\ref{pk}) is valid.
\end{theorem}
Further, Alanzi, Gagola, and Munagi \cite{AGM17} demonstrated that their
combinatorial proof of the Bessenrodt--Ono inequality can,
in principle,
be
adapted to the $k$-regular case. They
provided a proof for $k=2$ and 
indicated that the number of cases
increase
rapidly for larger $k$.

We extend the table of Beckwith and Bessenrodt for $
E_k$ and $
F_k$ to
$7 \leq k \leq 10$. This leads to
Table \ref{ausnahmetabelle7-10}.

\begin{table}[H]
\begin{equation*}
\begin{array}{|r|l|l|}
\hline
k & \text{Elements of }E
_k & \text{Elements of } F
_k \\ \hline
7 & (2,5), (2,7), (3,5) & (2,2), (2,3), (2,4), (3,3) \\
8 & (3,4), (3,5) & (2,2), (2,3), (2,4), (2,5), (3,3)\\
9 & (2,6), (3,4) & (2,2), (2,3), (2,4), (2,5), (3,3), (3,5) \\
10 & (2,6), (2,7), (3,4) & (2,2), (2,3), (2,4), (2,5), (3,3), (3,5)\\
\hline
\end{array}
\end{equation*}
\caption{\label{ausnahmetabelle7-10}
Exceptions to the Bessenrodt--Ono inequality for $7\leq k\leq 10$}
\end{table}

We observe that $
E_{
k}=E
_{\infty}$ and 
$F_{
k}=F
_{\infty}$ for $k=10$
and prove that this
even holds
true for all $k \geq 10$.

\begin{theorem}\label{Hauptresultat}
Let $a,b,k \geq 2$ be integers,
and $p_k(n)$ be
the $k$-regular partition function.
Then the Bessenrodt--Ono inequality
\begin{equation*}
p_k(a) \, p_k(b) > p_k(a+b) 
\end{equation*}
holds true for all pairs $(a,b)$ except for pairs in $
E_k$, where equality is in place and $F
_k$, where $p_k(a) \, p_k(b)< p_k(a+b)$. For $2 \leq k \leq 10$,
the sets $
E_k$ and $F
_k$ are given by
Tables \ref{ausnahmetabelle}
and \ref{ausnahmetabelle7-10}.
For all $k \geq 10$ we have 
\begin{equation*}
E_k = E
_{\infty} \text{ and }
F_k = F
_{\infty}.
\end{equation*}
\end{theorem}
\section{Proof of Theorem \ref{Hauptresultat}}
We closely follow
the proof scheme (\cite{HN22}, Section $2$)
proposed to prove
Bessenrodt--Ono inequalities. 

First, we estimate the arithmetic function $g_k(n)$ and the $k$-regular partition 
function $p_k(n)$ with the following upper and lower bounds,
independent of $k$.
\begin{lemma}
Let $k \geq 2$. Then
\begin{eqnarray}
g_{k}(n) & \leq & n \,  \big( 1 + \func{ln}(n) \big) \nonumber
, \\
p_k(n) &
>&
2^{
\sqrt{2n/3+1/4}-
3/2
}.
\label{easy}
\end{eqnarray}
\end{lemma}
\begin{proof}

The upper bound for $g_k(n) \leq \sigma(n)$ follows easily by integral comparison with the
bound for $n>1$. The lower bound of $p_k(n)$ cannot
be directly
related to $p(n)$ as
given in (\cite{HN22}, equation (2.2)), and
requires some
additional investigation.

The lower bound of $p_k(n)$ follows from 
\begin{equation*}
p_{k}\left( n\right) \geq
4^{
\left\lfloor
\sqrt{2n/3+1/2}-1
/2\right\rfloor
-\left\lfloor \left\lfloor
\sqrt{2n/3+1/4}-1
/2
\right\rfloor /k\right\rfloor }
>4
^{\left( k-1\right) \left(
\sqrt{2n/3+1/4}-
3/2
\right) /k}
. \end{equation*}
Here $\lfloor x \rfloor$ denotes the usual floor symbol.
Consider partitions of $n$ using only parts
from
$\left\{ 1,2,\ldots ,\left\lfloor \sqrt{n/3}\right\rfloor \right\} \setminus k \,\, \mathbb{N}$
with multiplicity at most $3$.
To
any admissible
submultiset, we can assign a partition.
The largest value we obtain is
\begin{eqnarray*}
3\sum _{m=1}^{\left\lfloor
\sqrt{2
n/3+1/4}-1
/2\right\rfloor }m&\leq &
\frac{3}{2}\left(
\sqrt{2
n/3+1/4}-1
/2+1\right) \left( \sqrt{2n/3+1/4}-1/2\right)
\\
&=&3\frac{2n/3+1/4-1/4}{2}=n
.
\end{eqnarray*}
Also, we have removed
$\left\lfloor \frac{1}{k}\left\lfloor \sqrt{\frac{2n}{3}+\frac{1}{4}}-\frac{3}{2}\right\rfloor \right\rfloor $
elements from the set.
This
demonstrates the claim.
\end{proof}

\begin{remark}
Note that the proof shows that we could even choose as a lower bound
$$p_k(n)
\geq
2
^{\left\lfloor
\sqrt{2n/3+1/4}-1/2\right\rfloor
}.$$
\end{remark}

\subsection{Proof of Theorem \ref{Hauptresultat} for $k>3$}

Let $A:=2$ and $B:=10$. Let $n \geq B$. We say the statement $S(n)$
holds true if for all partitions
$n=a +b$ with $a,b \geq A$:
\begin{equation*}\label{pab}
\Delta_k(a,b
):= p_{k}(a) \, p_{k}(b) - p_{k}(a+b)>0.
\end{equation*}

We assume that
$n >N_0>1$ and $S(m)$ hold
true for all $ B \leq m \leq n-1$. For $ B \leq
m\leq N_0$,
we demonstrate
$S(m
)$ through
a direct computer calculation with PARI/GP.
It will turn out that we can choose $N_{0}=
2938$.

We
obtain the following expressions
for $\Delta
_{k}\left(a,b
\right) =L+R$:
\begin{eqnarray*}\label{startl}
L &:=& -
\sum_{
\ell =1}^{b}
\frac{g _{k}\left(
\ell +a \right)}{a+b}
p_{k}\left(b-
\ell \right) ,\\
R &:=&
\sum _{
\ell =1}^{a}
\Big(
\frac{g_{k}(
\ell )}{a}\,
p_{k}(a-
\ell ) \, 
p_{k}(b) -\frac{g_{k}(
\ell )}{a+b} \,
p_{k}(a+b-
\ell )\Big) .
\label{startr}
\end{eqnarray*}
Further, we will refine the right sum $R$ into 
$$R=R_1+R_2+R_3.$$
\subsubsection{Left sum $L$}

We have
$
L
> -b \, \, p_{k}\left( b \right) \, 
\big( 1 + \ln (a+b)\big)  .
$

\subsubsection{Right sum $R
$}

Let 
\begin{equation*}
\ell _0:= a-\max \left\{ B -b,  A \right\} + 1. 
\end{equation*}
Therefore, $\ell
_0=
a-\max \left\{ 9-b, 1 \right\}$.
Let
\begin{eqnarray*}\label{decomposition}
R_{1}&:=&\sum_{\ell
=1}^{1}  \,f_{
\ell }(a,b)       , \,\, 
R_{2}:=\sum_{
\ell =2}^{\ell
_0-1} \,f_{
\ell }(a,b) , \,\, 
R_{3}:=\sum_{
\ell =\ell
_0}^{a} \,f_{
\ell }(a,b) ,
\\
\label{fkab}
f_{\ell
}(a,b)&:=&
\frac{g_{k}(
\ell )}{a}\,
p_{k}(a-
\ell ) \, 
p_{k}(b) -\frac{g_{k}(
\ell )}{a+b} \,
p_{k}(a+b-
\ell ) .
\end{eqnarray*}
\paragraph{The sum $R_{1}$}

We obtain
the lower bound:
\begin{equation*}
R_{1}
>
 \frac{b}{2a^{2}}\, p_{k}\left( a-1\right) p_{k}\left( b\right) .
\end{equation*}
\paragraph{The sum $R_{2
}$}
The second sum:
$R_2 > 0$.
\paragraph{The sum $R_{3}$} \ \newline
We split the third sum again into three parts: $R_3= R_{
31}+R_{
32}+R_{
33}$, where
\begin{eqnarray}
R_{31}:=
\sum_{
\ell =\ell
_0}^{a-A}   \, f_{
\ell }(a,b)
& > &
5 \, \left( 4 - p_{\infty }\left( 8\right) \right) \, \left(1+\ln \left( a\right) \right)  \label{a-1},
\label{big}\\
R_{32}:=\sum
_{\ell
=a-A+1}^{a-1}    \, f_{\ell
}(a,b)
& > &
- p_{k}\left( b\right)  \, \big( 1 + \ln \left( a\right) \big) ,\nonumber
\\
R_{33}:= \sum_{
\ell =a}^{a}  \, f_{\ell
}(a,b)  & \geq & 0 .\nonumber
\end{eqnarray}
We first use that
$p_{k}\left( 9-b\right) p_{k}\left( b\right) \geq p_{k}\left( 9\right) $,
which leads to at most $5$ summands.
The
estimation in (\ref{big}) follows from $p_{k}(a-
\ell ) \geq p_{k}(2)$. 
For the
estimation, we use that
\begin{equation*}
p_{k}\left( b \right) \geq p_{k} \left( 2 \right) = 2, \,\,
\frac{1}{a+b}<\frac{1}{a}, \text{ and } p_{k}(a+b-
\ell ) \leq p_{k}\left( 8\right) \leq p_{\infty }\left( 8\right) .
\end{equation*}
We obtain
\begin{equation*}
R_3 > 
-23 \, b \,\, p_{k}\left( b\right) \left(
1+\ln \left( a+b \right) \right).
\end{equation*}
\subsubsection{Final step}
Putting everything together leads to:
\begin{equation}
\Delta _{k}\left( a,b
\right)
>
\frac{ b \, p_{k} \left( b \right)}{ 2 \, a^2}  
\left( -48 \, a^2 \, \left( 1+\ln \left( 2a\right) \right) +
2^{\sqrt{2\left( a-1\right)
/3+1/4}-3/2}\right). \label{final}
\end{equation}
As the final step, we used the property (\ref{easy}). For $a\rightarrow \infty $
we can immediately observe that this is positive
since the second
term grows
in~$a$
faster than
$a\mapsto a^{2}\left( 1+\ln \left( 2a\right) \right) $.
In fact,
the expression (\ref{final}) is positive for all $ a \geq 1470
$. Note that if $a<1470
$, then $a+b\leq 2a
\leq N_0$.

Note that for $k
>N_{0}$, we have
$p_{k}\left( n\right) =p_{\infty}  \left( n\right) $ for
all $n\leq N_{0}
$. Therefore, we only have
to check
the inequality for
$k\leq N_{0}
$ and $n\leq
N_{0}$.
Hence,
$
\Delta _{k}\left( a,b
\right) >0$, which proves the Theorem.

\begin{remark}
Note that we can also prove the cases
$k\in \left\{ 2,3\right\} $ directly with
our method from \cite{HN22},
but then more adjustments are necessary as we cannot
directly use
(\cite{HN22}, (2.8)) anymore. For $k=2$,
we can choose $A=3$ and
$B=22$ with $N_{0}=3662$,
and for $k=3$,
we can choose $A=2$ and $B=17$
with $N_{0}=3776$. The crucial step is to
adjust the estimate of $R_{31}$
in (\ref{big}) for these
cases.
\end{remark}
\section{Log-concavity}
In this final
section, we provide further evidence, that the $k$-regular partition
function and the partition function $p_{\infty}(n)$ share
similar properties up to small $k$s.
We indicate
that,
as in the case of the Bessenrodt--Ono inequality,
this may also apply to
the log-concave property. 
Nicolas \cite{Ni78} and DeSalvo and Pak \cite{DP15} proved that the partition
function $p_{\infty}(n)$ is log-concave at $n$:
\begin{equation*}
p(n)_{\infty}^2 \geq p_{\infty}(n-1) \, p_{\infty}(n+1)
\end{equation*} 
if and only if $n \geq 25$ or $n$ is even. For $k$-regular partitions,
building on Hagis' Rademacher-type formula and a new method developed by
Griffin, Ono, Rolen,
and Zagier \cite{GORZ19},
Craig and Pun \cite{CP20} proved that the $k$-regular partition functions $p_k(n)$ satisfy higher-order Tur\'an inequalities for sufficiently large $n$ dependent on $k$. Let $N_k$ be the smallest number, such that $\{p_k(n)\}_n$ is log-concave
at $n$
\begin{equation*}
p_k(n)^2 \geq p_k(n-1) \, p_{k}\left( n+1\right)
\end{equation*}
for all $n \geq N_k$. We calculated all exceptions for $2 \leq k \leq 100$ and
$1 \leq n \leq 1000$. For $2 \leq k \leq 20$,
we have provided them
in Table \ref{ex}, where for $k=3$, the
exceptions are all
odd $1
\leq n \leq 57$. This is compatible with the numbers conjectured by Craig and Pun (\cite{CP21}, Remark 1.3) for $2 \leq k \leq 4$.
Recently, Dong and Ji \cite{DJ24} proved the conjecture for $k=2$.

Similar to the Bessenrodt--Ono inequality, we believe
that the set
of exceptions stabilizes
for large $k$
and is the same as the one of $p_{\infty}(n)$.
This set is given by all odd numbers $1 \leq n \leq 25$. We put $N_{\infty}:= 26$.
\begin{conj}
Let $k \geq 2$. Let $N_k$ be the smallest integer, such that $\{p_k(n)\}_n$ is log-concave
at every
$n \geq N_k$. Then
$N_k = N_{\infty}$ if $k\geq 30$.
\end{conj}
This is in accordance with
Remark 1.3 given by Craig and Pun \cite{CP21}.
Since $p_k(n)=p_{\infty}(n)$ for $k >n$ we
also have:

\begin{corollary}
Let the conjecture hold
true. Let $k \geq 30$. Then the exceptions $n$ are 
the same as for
$\left\{ p_{\infty}(n)\right\} _{n}$.
\end{corollary}
\begin{table}
\[
\begin{array}{rrrrrrrrrrrrrrrrrrrr}
\hline \hline
n\backslash k&2&3&4&5&6&7&8&9&10&11&12&13&14&15&16&17&18&19&20\\ \hline
1&&\bullet &\bullet &\bullet &\bullet &\bullet &\bullet &\bullet &\bullet &\bullet &\bullet &\bullet &\bullet &\bullet &\bullet &\bullet &\bullet &\bullet &\bullet \\
2&\bullet &&&&&&&&&&&&&&&&&&\\
3&&\bullet &&\bullet &\bullet &\bullet &\bullet &\bullet &\bullet &\bullet &\bullet &\bullet &\bullet &\bullet &\bullet &\bullet &\bullet &\bullet &\bullet \\
4&\bullet &&\bullet &&&&&&&&&&&&&&&&\\
5&&\bullet &&\bullet &\bullet &\bullet &\bullet &\bullet &\bullet &\bullet &\bullet &\bullet &\bullet &\bullet &\bullet &\bullet &\bullet &\bullet &\bullet \\
6&&&&&&&&&&&&&&&&&&&\\
7&&\bullet &&\bullet &\bullet &\bullet &\bullet &\bullet &\bullet &\bullet &\bullet &\bullet &\bullet &\bullet &\bullet &\bullet &\bullet &\bullet &\bullet \\
8&\bullet &&\bullet &&&&&&&&&&&&&&&&\\
9&&\bullet &&\bullet &\bullet &\bullet &&\bullet &\bullet &\bullet &\bullet &\bullet &\bullet &\bullet &\bullet &\bullet &\bullet &\bullet &\bullet \\
10&&&&&&&&&&&&&&&&&&&\\
11&\bullet &\bullet &\bullet &\bullet &\bullet &\bullet &\bullet &\bullet &\bullet &\bullet &\bullet &\bullet &\bullet &\bullet &\bullet &\bullet &\bullet &\bullet &\bullet \\
12&&&&&&&&&&&&&&&&&&&\\
13&\bullet &\bullet &\bullet &\bullet &&\bullet &\bullet &\bullet &\bullet &\bullet &\bullet &\bullet &\bullet &\bullet &\bullet &\bullet &\bullet &\bullet &\bullet \\
14&\bullet &&&&&&&&&&&&&&&&&&\\
15&&\bullet &&\bullet &&\bullet &&\bullet &\bullet &\bullet &\bullet &\bullet &\bullet &\bullet &\bullet &\bullet &\bullet &\bullet &\bullet \\
16&\bullet &&\bullet &&&&&&&&&&&&&&&&\\
17&\bullet &\bullet &&\bullet &&\bullet &\bullet &\bullet &\bullet &\bullet &\bullet &\bullet &\bullet &\bullet &\bullet &\bullet &\bullet &\bullet &\bullet \\
18&&&&&&&&&&&&&&&&&&&\\
19&\bullet &\bullet &&\bullet &&\bullet &\bullet &\bullet &\bullet &\bullet &\bullet &\bullet &\bullet &\bullet &\bullet &\bullet &\bullet &\bullet &\bullet \\
20&\bullet &&&&&&&&&&&&&&&&&&\\
21&&\bullet &&\bullet &&\bullet &&\bullet &&\bullet &\bullet &\bullet &&\bullet &\bullet &\bullet &\bullet &\bullet &\bullet \\
22&&&&&&&&&&&&&&&&&&&\\
23&\bullet &\bullet &&\bullet &&\bullet &&\bullet &&\bullet &&\bullet &\bullet &\bullet &\bullet &\bullet &\bullet &\bullet &\bullet \\
24&&&&&&&&&&&&&&&&&&&\\
25&&\bullet &&\bullet &&\bullet &&\bullet &&\bullet &&\bullet &&\bullet &\bullet &\bullet &\bullet &\bullet &\bullet \\
26&\bullet &&&&&&&&&&&&&&&&&&\\
27&&\bullet &&\bullet &&\bullet &&\bullet &&\bullet &&\bullet &&\bullet &&\bullet &&\bullet &\\
28&&&&&&&&&&&&&&&&&&&\\
29&\bullet &\bullet &&\bullet &&\bullet &&\bullet &&\bullet &&\bullet &&\bullet &&\bullet &&\bullet &\\
30&&&&&&&&&&&&&&&&&&&\\ 
31&&\bullet &&\bullet &&\bullet &&\bullet &&\bullet &&\bullet &&&&&&&\\
32&\bullet &&&&&&&&&&&&&&&&&&\\
33&&\bullet &&\bullet &&\bullet &&\bullet &&&&&&&&&&&\\
34&&&&&&&&&&&&&&&&&&&\\
35&&\bullet &&\bullet &&\bullet &&&&&&&&&&&&&\\
36&&&&&&&&&&&&&&&&&&&\\
37&&\bullet &&\bullet &&&&&&&&&&&&&&&\\
38&&&&&&&&&&&&&&&&&&&\\
39&&\bullet &&\bullet &&&&&&&&&&&&&&&\\
40&&&&&&&&&&&&&&&&&&&\\
41&&\bullet &&\bullet &&&&&&&&&&&&&&&\\
42&&&&&&&&&&&&&&&&&&&\\
43&&\bullet &&&&&&&&&&&&&&&&&\\
44&&&&&&&&&&&&&&&&&&&\\
45&&\bullet &&&&&&&&&&&&&&&&&\\
\hline \hline
\end{array}
\]
\caption{\label{ex}Exceptions of log-concavity at $n$ for $2\leq k\leq 20$ and $1\leq n\leq 45$}
\end{table}
\end{document}